\documentclass[12pt, reqno, academicons]{amsart}
\usepackage{amsmath, amsthm, amscd, amsfonts, amssymb, graphicx, color}
\usepackage[bookmarksnumbered, colorlinks, plainpages]{hyperref}
\hypersetup{colorlinks=true,linkcolor=red, anchorcolor=green, citecolor=cyan, urlcolor=red, filecolor=magenta, pdftoolbar=true}

\usepackage{tensor}

\newtheorem{theorem}{Theorem}[section]
\newtheorem{lemma}[theorem]{Lemma}

\newtheorem{corollary}[theorem]{Corollary}
\theoremstyle{definition}
\newtheorem{definition}[theorem]{Definition}
\newtheorem{example}[theorem]{Example}

\newtheorem{conjecture}[theorem]{Conjecture}

\newtheorem{problem}[theorem]{Problem}
\theoremstyle{remark}
\newtheorem{remark}[theorem]{Remark}
\numberwithin{equation}{section}

\usepackage{fullpage}

\begin{document}

\setcounter{page}{1}

\title[FPT]{ Chatterjea type fixed point in Partial $b$-metric spaces}

\author[Ya\'e Ulrich Gaba]{Ya\'e Ulrich Gaba$^{1,3,\dagger}$}

\author{Collins Amburo Agyingi$^{2,3}$}

\author[Domini Jocema Leko]{Domini Jocema Leko$^{3}$}

\address{$^{1}$Institut de Math\'ematiques et de Sciences Physiques (IMSP)/UAC, 01 BP 613
	Porto-Novo, B\'enin.}

	\address{$^{2}$ Department of Mathematics and Applied Mathematics, Nelson Mandela University,
	P.O. Box 77000, Port Elizabeth 6031, South Africa.}

\address{$^{3}$African Center for Advanced Studies, P.O. Box 4477, Yaounde, Cameroon.}

\address{$^{\dagger}$Corresponding author.}

\vspace{0.3cm}

\email{\textcolor[rgb]{0.00,0.00,0.84}{yaeulrich.gaba@gmail.com
}}

\email{\textcolor[rgb]{0.00,0.00,0.84}{collins.agyingi@mandela.ac.za
}}

\email{\textcolor[rgb]{0.00,0.00,0.84}{domini.leko@gmail.com
}}

\subjclass[2010]{Primary 47H05; Secondary 47H09, 47H10.}

\keywords{partial
	$b$-metric; fixed point}

\begin{abstract}
	In this paper, we give and prove two Chatterjea type
	fixed point theorems on partial $b$-metric space. We propose an extension to the Banach contaction principle on partial $b$-metric space which was already presented by Shukla and also study some related
	results on the completion of a partial metric type space. In particular, we prove a joint Chatterjea-Kannan fixed point theorem. We verify the $T$-stability of
	Picard's iteration and conjecture the $P$ property for such maps. We also give
	examples to illustrate our results.

\end{abstract}

\maketitle

\section{Introduction and Preliminaries}

In literature, one finds numerous generalizations of metric spaces and Banach contraction principle (BCP). In this line, Czerwik \cite{cze} proposed $b$-metric
spaces as a generalization of metric spaces and proved the famous BCP in such spaces. In this sequel, Gaba \cite{gaba} introduced the so-called ``metric type space" and proved a common fixed point theorem with the help what he called $\lambda$-sequence in that setting. 
 
After Matthews \cite{mat} introduced partial metric spaces as a generalization of the
metric space, many authors have studied fixed point theorems on theses spaces (e.g. \cite{abdel,oltra}), in particular, Shukla\cite{shu}
gave some analog of the Banach contraction principle as well as the Kannan type fixed point theorem
in partial $b$-metric spaces. 

 In this paper, analogs of the Chatterjea fixed point theorem are proved. 
 
First, we recall some definitions from partial $b$-metric spaces.

\begin{definition} (Compare \cite{mat})
	A partial metric type on a set $X$ is a function $p: X \times X \to [0, \infty)$ such that:
	\begin{enumerate}
		\item[(pm1)] $x = y$ iff $(p(x, x) = p(x, y)=p(y,y)$ whenever $x, y \in X$,
		\item[(pm2)] $0\leq p(x, x)\leq p(x, y)$ whenever $x, y \in X$,
		
		\item[(pm3)] $p(x, y) = p(y, x);$ whenever $x, y \in X$,
		
		\item[(pm4)] There exists a real number $s\geq 1$ such that
		$$p(x, y) + p(z,z)\leq  s[p(x,z)+p(z,y)]$$ 
		for any points $x,y,z\in X$. 
		
	\end{enumerate}
	
The pair $(X, p)$ is called a partial metric type space or a partial $b$-metric space.
	\end{definition}

It is clear that, if $p ( x , y ) = 0$ , then, from (pm1) and (pm2), $x = y$.

The family $\mathcal B'$ of sets
\begin{equation}\label{def.basis2-pm}
B'_p(x,\varepsilon):=\{y\in X : p(x,y)<\varepsilon +p(x,x)\},\; x\in X,\, \varepsilon>0\,,\end{equation}
is a basis for a topology $\tau(p)$ on $X$. The topology $\tau(p)$ is $T_0$.

\begin{definition}
	Let $(X, p)$ be a partial $b$-metric space. Let $(x_n)_{n\geq 1}$ be any sequence in $X$ and $x \in X$. Then:
	
	\begin{enumerate}
		\item The sequence $(x_n)_{n\geq 1}$ is said to be convergent with respect to $\tau(p)$ (or $\tau(p)$-convergent) and converges to $x$, if $\lim\limits_{n\to \infty} p(x , x_n) = p(x, x)$.
		We write $$x_n \overset{p}{\longrightarrow} x.$$

		\item The sequence $(x_n)_{n\geq 1}$ is said to be a $p$-Cauchy sequence if
		$$\lim\limits_{n\to \infty,m\to \infty} p(x_n , x_m)$$ exists and is finite.

	\end{enumerate}

	\ $(X, p)$ is said to be complete if for every $p$-Cauchy sequence $(x_n)_{n\geq 1} \subseteq X$, there exists $x \in X$ such that:
	$$ \lim\limits_{n\to \infty,m\to \infty} p(x_n , x_m)= \lim\limits_{n\to \infty} p(x , x_n)=p(x,x).$$
	
\end{definition}

We give these additional definitions, useful to characterize some specific complete partial metric type spaces.
\begin{definition}
	Let $(X, p)$ be a partial $b$-metric space.
	\item The sequence $(x_n)_{n\geq 1}\subset X$ is called $0$-Cauchy if
	$$\lim\limits_{n,m\to \infty} p(x_n , x_m)=0.$$ 
	
	\item $(X, p)$ is called $0$-complete  if for every $0$-Cauchy
	sequence $(x_n)_{n\geq 1} \subseteq X$, there exists $x \in X$ such that:
	
	$$ \lim\limits_{n,m\to \infty} p(x_n , x_m)= \lim\limits_{n\to \infty} p(x_n , x)=p(x,x)=0.$$
\end{definition}

\section{BCP extension}

In this section, we show that if $T$ is a self-map on a partial metric space type space $(X,p)$ and has a power which is a contraction, i.e. there exists $n\in \mathbb{N}, n>1$ and $0\leq \lambda<1$ such that 

\[p(T^nx,T^ny) \leq \lambda p(x,y) ,\]

then there is a transformation $p'=\phi(p)$ of $p$ such that $T$ a contraction on $(X,p'))$. Moreover, we prove that the partial metric type space $(X,p')$ is $0$-complete if $T$ is uniformly continuous.

Ideas for this section are merely copies of the results presented in \cite{gabs}. We adjust them in the partial metric type setting.
We begin with the following definitions.

\begin{definition} 
	Two partial metrics type $p_1$ and $p_2$ on a set $X$ are said to be equivalent if there exist $\alpha,\beta\geq 0$  such that
	$$\alpha p_1(x,y) \leq p_2(x,y) \leq \beta p_1(x,y)  , \ \text{for all } x,y,z \in X.$$
	
\end{definition}

\begin{definition}
	Given two partial metric type spaces $(X,p_1)$ and $(Y,p_2)$ , we say that $T : (X,p_1) \to (Y,p_2)$ is uniformly continuous if for every real number $\varepsilon > 0$ and $\lambda>0$ there exists $\delta=\delta(\lambda) > 0$ such that for every $x, y \in X$ with $p_1(x, y) < \delta$, we have that $p_2(Tx, Ty) < \varepsilon$.
	
\end{definition}

\begin{theorem}\label{BCP}(\cite[Theorem 1.]{shu})
	
	Let $(X,p)$ be a complete partial $b$-metric space with coefficient
	$s \geq 1$ and let $T:X\to X$
	be a mapping such that there exists $\lambda\in [0,1)$ satisfying
	
	\begin{equation}\label{BCPeq}
	p(Tx,Ty) \leq \lambda \ p(x,y),
	\end{equation}

	whenever $x,y\in X.$ 
	Then $T$ has a unique fixed point.
	% In fact, $T$ is a Picard operator.
\end{theorem}

We give the following natural corollary:

\begin{corollary}\label{BCPcor}
	
	Let $(X,p)$ be a complete partial metric type space and let $T:X\to X$
	be a mapping such that there exists $\lambda\in [0,1)$ satisfying
	
	\[
	p(T^nx,T^ny) \leq \lambda \ p(x,y),
	\]
	for some $n>1$, whenever $x,y\in X.$ 
	Then $T$ has a unique fixed point.
\end{corollary}

\begin{proof}
	By Theorem \ref{BCP}, $T^n$ has a unique fixed point, say $x\in X$ with $T^nx = x$. Since
	
	\[
	T^{n+1}x = T(T^nx)=Tx = T^n(Tx),
	\]
	it follows that $Tx$ is a fixed point of $T^n$, and thus, by the uniqueness of $x$, we have
	$Tx = x$, that is, $T$ has a fixed point.  Since, the fixed point of $T$ is necessarily a fixed point of $T^n$, so it is unique.
\end{proof}

The main theorem of this section is as follows:

\begin{theorem}\label{res1}
	
	Let $d$ be a partial metric type on a space $X$ and $T : (X,p) \to (X,p)$ a self mapping such that:
	
	\[
	p(T^nx,T^ny) \leq  K p(x,y),
	\]
	
	for some $n>1$ and $0<K<1$, whenever $x,y,z\in X.$ If $\lambda$ is a nonnegative real such that $$K^{\frac{1}{n}}<\frac{1}{\lambda}<1,$$
	
	then the application $p':X^2\to [0,\infty)$ defined by :
	\[
	p'(x,y) = \sum_{i=0}^{n-1}\lambda^i p(T^ix,T^iy), \text{ whenever } x,y\in X,
	\]
	satisfies: 
	
	\begin{itemize}
		\item[i)] $p'$ is a  partial metric type on the space $X$;
		
		\item[ii)] $T : (X,p') \to (X,p')$ a self mapping such that:
		\[
		p'(Tx,Ty) \leq  \frac{1}{\lambda} p'(x,y).\footnote{i.e. $T$ is a contraction with constant $\frac{1}{\lambda}$ with respect to $p'$.}
		\]
		
	\end{itemize}
	
\end{theorem}

\begin{proof}
	We first prove that $p'$ is a partial metric type:
	\begin{enumerate}
		\item[(pm1)] Indeed for $x,y \in X,$ if $x=y$, then $$p'(x,y)=p'(x,x)=p'(y,y).$$
		Conversely, assume $x,y \in X,$ are such that $p'(x,y)=p'(x,x)=p'(y,y)$, which means 
		
		\[ \sum_{i=0}^{n-1}\lambda^i p(T^ix,T^iy) = \sum_{i=0}^{n-1}\lambda^i p(T^ix,T^ix) = \sum_{i=0}^{n-1}\lambda^i p(T^iy,T^iy).  \] 
		
		 It is therefore obvious that $$p(T^ix,T^ix)=p(T^iy,T^ix) =p(T^iy,T^iy) \text{ for } i=0,\cdots,n-1,$$  in particular $p(x,x)=p(y,y)=p(x,y)$, i.e. $x=y.$

	\vspace*{0.3cm}
		
		\item[(pm2)] For all $x,y\in X$  and for all $i=0,\cdots,n-1,$ we have

		$$0\leq p(T^ix,T^ix) \leq p(T^ix,T^iy) ,$$
		and hence 
		\[ \sum_{i=0}^{n-1}\lambda^i p(T^ix,T^ix) \leq \sum_{i=0}^{n-1}\lambda^i p(T^ix,T^iy) \]
i.e. $$p'(x,x) \leq p'(x,y).$$

		\item[(pm3)] For all $x,y\in X$, 
		
		\[ p'(x,y) = \sum_{i=0}^{n-1}\lambda^i p(T^ix,T^iy)=\sum_{i=0}^{n-1}\lambda^i p(T^iy,T^ix)=p'(y,x) ,   \]
		
	that is $$ p'(x,y) = p'(y,x) $$
		
	for all $x,y\in X.$	%we have that
		
		%$$0<d'(x,x,y)\leq d(x,x,y)\leq d(x,y,z)\leq d'(x,y,z).$$
		
%		\item[(G4)] Trivially $d'(x,y,z)= d'(x,z,y)=d'(y,z,x)=\ldots$, (symmetry in all three variables).
		
		\item[(pm4)] For all $x,y,a\in X$, since 
		
		$$\lambda^i [p(T^ix,T^iy)+p(T^ia,T^ia) ]\leq \lambda^is [p(T^ix,T^ia)+ p(T^ia,T^iy)], $$
		
		we get 
		
		\begin{align*}
		p'(x,y) & = \sum_{i=0}^{n-1}\lambda^i p(T^ix,T^iy)\\
		& \leq \sum_{i=0}^{n-1}\lambda^is [p(T^ix,T^ia)+ p(T^ia,T^iy)]  - \sum_{i=0}^{n-1} \lambda^i p(T^ia,T^ia) \\
		& = s\sum_{i=0}^{n-1}\lambda^i [p(T^ix,T^ia)+ p(T^ia,T^iy)]  - \sum_{i=0}^{n-1}\lambda^i p(T^ia,T^ia)\\
		& =s[p'(x,a)+p'(a,y)] - p'(a,a).
		\end{align*}
So 
\[ p'(x,y)+p'(a,a) \leq s[p'(x,a)+p'(a,y)]  \]	

for any $x,y,a\in X.$	
		
	\end{enumerate}
	
	Hence, $p'$ is a partial metric type space on $X$.

\vspace*{0.3cm}
	
	We now prove that $T : (X,p') \to (X,p')$ is a contraction with constant $\frac{1}{\lambda}$.
	
	It is readily seen, by a simple computation, that 
	\[
	p'(Tx,Ty) = \frac{1}{\lambda} [p'(x,y)-p(x,y)]+\lambda^{n-1}p(T^nx,T^ny).
	\]
	Since  $T^n : (X,p) \to (X,p)$ is a contraction with constant $K$, it follows that
	
	\begin{align*}
	p'(Tx,Ty) & \leq \frac{1}{\lambda} [p'(x,y)-p(x,y)]+K\lambda^{n-1}p(x,y)\\
	& = \frac{1}{\lambda}p'(x,y)+\left( K-\frac{1}{\lambda^n}\right)\lambda^{n-1}p(x,y)\\
	& \leq \frac{1}{\lambda}p'(x,y),
	\end{align*}
	because of the choice $K^{\frac{1}{n}}<\frac{1}{\lambda}.$
	This completes the proof.
\end{proof}

As observed in \cite[Remark 2.2]{gabs}, under the assumptions of Theorem \ref{res1}, it is readily seen that

\begin{align*}
p'(x,y) & \leq \sum_{i=0}^{\infty}\lambda^i p(T^ix,T^iy)\\
& \leq  p'(x,y) + \lambda^n K p'(x,y) + \lambda^{2n} K^2p'(x,y)+\cdots \\
&= \frac{1}{1-\lambda^nK}p'(x,y).
\end{align*}

The term $h(x,y):=\sum_{i=0}^{\infty} \lambda^i p(T^ix,T^iy)$ therefore defines a partial metric type, equivalent to $p'$, as long as the series happen to converge for some $\lambda>1$.

\vspace*{0.3cm}

Next, we establish that whenever the mapping $T:(X,p)\to (X,p)$ is uniformly continuous and the partial metric type
$p$ is $0$-complete, then the partial metric type $p'$ is also $0$-complete.

\begin{theorem}\label{res2}
	We repeat the assumptions of Theorem \ref{res1}. If $T$ is uniformly continuous and the partial metric type $p$ is $0$-complete, then so is the partial metric type $p'$.
	
\end{theorem}

\begin{proof}
	
	Since $p(x,y)\leq p'(x,y)$ for any $x,y\in X$, any $0$-Cauchy sequence in $(X,p')$ is also a $0$-Cauchy sequence in $(X,p)$. It is therefore enough to prove that, under uniform continuity of $T$, any convergent sequence $(x_n)_{n\geq 1} \subset(X,p)$
	such that there exists $x^*\in X$ with $$\lim\limits_{n\to \infty} p(x_n,x^*) = \lim\limits_{n,m\to \infty} p(x_n,x_m)=p(x^*,x^*)=0,$$ is such that 
there exists $y^*\in X$ with $$\lim\limits_{n\to \infty} p'(x_n,y^*) = \lim\limits_{n,m\to \infty} p'(x_n,x_m)=p'(y^*,y^*)=0.$$
	
	\vspace*{0.5cm}
	
	So let $(x_n)_{n\geq 1}$ be a sequence in the $G$-metric space $(X,p)$ such that $(x_n)_{n\geq 1}$ converges to some $\xi \in (X,p).$ and $p(\xi,\xi)=0.$ Set $M=\max\{\lambda^i, i=1,\cdots, n-1\}$ and observe that $$M\geq \lambda>1.$$
	Since all the powers of $T$ are also uniformly continuous in $(X,d)$, we can write that, for any $\varepsilon>0$, there exists $\eta >0$ such that for all $x,y \in X,$ and $i=1,\cdots, n-1$
	
	\[
	p(x,y) <\eta \Longrightarrow p(T^ix,T^iy)<\frac{\varepsilon}{Mn}.
	\]
	Since $\{x_n\}$ converges to some $\xi \in (X,p),$ and $p(\xi,\xi)=0$ there exists $n_0\in \mathbb{N}$ such that $$k\geq n_0 \Longrightarrow p(\xi,x_k)<\eta.$$
	
	Then 
	$$k>n_0 \Longrightarrow p(T^i\xi,T^ix_k)<\frac{\varepsilon}{Mn} \ \text{ for }i=1,\cdots, n-1,$$
	i.e.
	
	$$   p'(\xi,x_k,) < \frac{\varepsilon}{n}\left[    \frac{1}{M}+\frac{\lambda}{M}+\cdots+ \frac{\lambda^{n-1}}{M}\right]<\varepsilon. $$
	Thus $(x_n)_{n\geq 1}$ converges to $\xi$ with respect to the partial metric space $p'$ and $p'(\xi,\xi)=0.$

This completes the proof.	
	
\end{proof}

In concluding this section, we introduce what we call partial ultra-metrics and conjecture that the construction of Frink\cite{f} could be used to obtain a modular metric
from an  ultra-modular metric.
Taking inspiration from the theory of ultra-metric space and that of metric type spaces (see \cite{gaba}), we can define:

\begin{definition}\label{def2}
	A partial ultra-metric
	on the set $X$ is is a function $p: X \times X \to [0, \infty)$ such that:
	\begin{enumerate}
		\item[(pm1)] $x = y$ iff $(p(x, x) = p(x, y)=p(y,y)$ whenever $x, y \in X$,
		\item[(pm2)] $0\leq p(x, x)\leq p(x, y)$ whenever $x, y \in X$,
		
		\item[(pm3)] $p(x, y) = p(y, x);$ whenever $x, y \in X$,
		
		\item[(pm4)] here exists a real number $s\geq 1$ such that
		$$p(x, y) + p(z,z)\leq  \max\{p(x,z)+p(z,y)\}$$ 
		for any points $x,y,z\in X$. 
		
	\end{enumerate}

	The pair $(X, p)$ is called a partial ultra-metric space .
\end{definition}

We are interested in the following question: 
\begin{problem}
	\textcolor{blue}{Given a partial ultra metric $\omega$ on a non empty set $X$, can we construct a partial metric type $\omega'$ on $X$ such that $\omega$ and $\omega'$ are equivalent? If not, are there conditions which guarantee the existence of such a partial metric type $\omega'$ on $X$?}
\end{problem}
The authors plan to take up this investigation \cite{leko} by using ``\textit{the chain construction}" as a tool.

\section{Main results}

In this section, we present some fixed point results for Chatterjea type mapping in the setting of a partial $b$-metric space. Following theorem is an analog to Chatterjea fixed point theorem in partial $b$-metric space.

\begin{theorem}\label{main}
Let $(X, p)$ be a complete partial $b$-metric space with coefficient
$s \geq  2$ and $T:X\to X$ be a mapping satisfying the following condition:
\begin{equation}\tag{Ch}\label{Chat}
p(Tx,Ty)\le \lambda  \,\left[p(x,Ty)+p(y,Tx)\right]\,.
\end{equation}

for all $x,y \in X$, where $\lambda\in \left[0,\frac{1}{s^2}\right)$. Then $T$ has a unique fixed point $u \in X$ and $p(u, u)=0$.
\end{theorem}

\begin{proof}
	Let us first show that if $T$ has a fixed point $u$, then it is unique and $p(u,u)=0$.
	
	From \eqref{Chat}, we have
	\[ p(u,u)=p(Tu,Tu) \leq \lambda[p(u,Tu)+ p(u,Tu)] = 2\lambda p(u,Tu) < p(u,u),\]
	a contradiction, unless $p(u,u)=0.$

Suppose $u, v \in X$ are two distinct fixed points of $T$, that is, $T u = u$, $T v = v$ and $u\neq v.$. Then it follows from \eqref{Chat} that

\begin{align*}
	p(u,v) = p(Tu,Tv) &\leq \lambda[p(u,Tv)+p(v,Tu)]\\ 
	                  &\leq 2\lambda p(u,v) < p(u,v)
\end{align*}	
a contradiction, unless $p(u,v)=0$, i.e. $u=v.$ Thus if a fixed point of $T$
exists, then it is unique.	
For existence of fixed point, let $x_0 \in X$ be arbitrary;
set $x_n = T^n x_0$ and $b_n = p(x_n , x_{n+1} )$. Without loss of generality, we may assume that $b_n>0$ for all $n\geq 0$ otherwise $x_n$ is a fixed point of $T$ for at least one $n \geq 0$.

For any $n \in \mathbb{N}$, it follows from \eqref{Chat} that	
	
\begin{align*}
	b_n = p(x_n , x_{n+1} ) &= p(T x_{n-1} , T x_n )\\
	                        & \leq \lambda [p(x_{n-1}, x_{n+1})+ p(x_n,x_n)] \\
	                        & \leq \lambda [p(x_{n-1}, x_{n}) + p(x_n,x_{n+1}) -p(x_n,x_n)+p(x_n,x_n)     ]\\
	                        & = \lambda [p(x_{n-1}, x_{n}) + p(x_n,x_{n+1})]\\
	                        & = \lambda[b_{n-1}+b_n],
\end{align*}	

therefore $b_n \leq \mu b_{n-1}$ where $\mu = \frac{\lambda}{1-\lambda} <1$ (since $\lambda\in \left[0,\frac{1}{s^2}\right)\subset \left[0,\frac{1}{2}\right) $). On repeating this, one obtains 

\begin{equation}\label{eq7}
	b_n \leq \mu^nb_0
\end{equation}

hence $\lim\limits_{n\to \infty}b_n =0.$
	
For $m,n \in \mathbb{N}$ with $m>n$, we obtain

\begin{align*}
p(x_n,x_m) &\leq s[p(x_n,x_{n+1}) + p(x_{n+1},x_m) ] - p(x_{n+1},x_{n+1}) \\
& \leq s p(x_n,x_{n+1}) +s^2[p(x_{n+1},x_{n+2})+p(x_{n+2},x_{m})] - s p(x_{n+2},x_{n+2})\\
&\leq s p(x_n,x_{n+1}) + s^2 p(x_{n+1},x_{n+2})+s^3p(x_{n+2},x_{n+2})\\
& + \cdots + s^{m-n}p(x_{m-1},x_m).
\end{align*}
	
Using \eqref{eq7} in the above inequality,	
	
	\begin{align*}
		p(x_n,x_m) &\leq s \mu^n [1+s\mu+(s\mu)^2+\cdots]p(x_0,x_1)\\
		&\leq \frac{s\mu^n}{1-s\mu}p(x_0,x_1).
	\end{align*}
	
As $\lambda\in \left[0,\frac{1}{s^2}\right) \subset \left[0,\frac{1}{s}\right)$	and $s>1$, it follows from the above inequality that
	
	$$\lim\limits_{n,m\to \infty}p(x_n,x_m) =0.$$
	
Therefore, $(x_n)$ is a Cauchy sequence in $X$. By completeness of X there exists $ x^*\in X$ such that	
	
	\begin{equation}\label{eq8}	
	\lim_{n\to\infty}p(x^*,x_n)=\lim_{n,m\to\infty} p(x_n,x_m)=p(x^*,x^*)=0. 
		\end{equation}

We shall show that $x^*$ is a fixed point of $T$.	

For any $n \in \mathbb{N}$ it follows from \eqref{Chat} that

\begin{align*}
	p(x^*,Tx^*) &\leq s[p(x^*, x_{n+1} ) + p(x_{n+1} , T x^*)] - p(x_{n+1} , x_{n+1} )\\
	&\leq s[p(x^*, x_{n+1} ) + p(Tx_{n} , T x^*)]\\
	& \leq s[p(x^*, x_{n+1} ) +\lambda (p(x_n,Tx^*)+p(x^*,x_{n+1}))]\\
	& \leq s p(x^*, x_{n+1} ) + s\lambda p(x^*,x_{n+1}) \\
	& + s^2\lambda[p(x_n,x^*)+p(x^*,Tx^*)] - s\lambda p(x^*,x^*).
\end{align*}

Taking limit as $n\to \infty$, as $p(x^*,x^*)=0,$ we have

\[ p(x^*,Tx^*)\leq s^2 \lambda p(x^*,Tx^*)< p(x^*,Tx^*) , \]

--a contradiction, unless $p(x^*,Tx^*)=0,$ that is, $Tx^* = x^*$.
Thus, $x^*$ is the unique fixed point of $T$.
\end{proof}

\begin{theorem}\label{main22}
Let $(X, p)$ be a complete partial $b$-metric space with coefficient
$s >  1$ and $T:X\to X$ be a mapping satisfying the following condition:
\begin{equation}\tag{Ch2}\label{Chat2}
p(Tx,Ty)\leq \lambda  \max\{p(x,y),p(x,Ty),p(y,Tx)\}\,.
\end{equation}

for all $x,y \in X$, where $\lambda\in \left[0,\frac{1}{s}\right)$. Then $T$ has a unique fixed point $u \in X$ and $p(u, u)=0$.
\end{theorem}

\begin{proof}
Let us first show that if $T$ has a fixed point $u$, then it is unique and $p(u,u)=0$.

Suppose $u, v \in X$ are two distinct fixed points of $T$, that is, $T u = u$, $T v = v$ and $u\neq v.$. Then it follows from \eqref{Chat2} that

\begin{align*}
p(x_{n+1},x_n) & = p(Tx_{n},Tx_{n-1})\\
               &  \leq \lambda \max\{ p(x_{n},x_{n-1}), p(x_n,x_n),p(x_{n-1},x_{n+1})  \} \\
               & \leq \max\{ p(x_{n},x_{n-1}), p(x_{n-1},x_{n+1})  \}
\end{align*}
since $p(x,x)\leq p(x,y)$ whenever $x,y\in X.$

\vspace*{0.3cm}

At this point, we distinguish between two cases.

\begin{enumerate}
\item[Case 1.] 
 $\max\{ p(x_{n},x_{n-1}), p(x_{n-1},x_{n+1})  \} =p(x_{n},x_{n-1}) $

 \[ p(x_{n+1},x_{n}) \leq \lambda p(x_{n-1},x_{n}). \]

Iterating this process, we get

\[ p(x_{n+1},x_{n}) \leq  \lambda^n p(x_{0},x_{1}), \]
for all $n \in \mathbb{N}$.

From the proof of the previous theorem, we can easily establish that for $m,n \in \mathbb{N}$ with $m>n$, 

	\begin{align*}
p(x_n,x_m)
\leq \frac{s\lambda^n}{1-s\lambda}p(x_0,x_1).
\end{align*}

% \left[0,\frac{1}{s^2}\right) \subset

As $\lambda\in \left[0,\frac{1}{s}\right)$	and $s>1$, it follows from the above inequality that

$$\lim\limits_{n,m\to \infty}p(x_n,x_m) =0.$$

Therefore, $(x_n)$ is a Cauchy sequence in $X$. By completeness of X there exists $ x^*\in X$ such that	

\begin{equation}\label{eq88}	
\lim_{n\to\infty}p(x^*,x_n)=\lim_{n,m\to\infty} p(x_n,x_m)=p(x^*,x^*)=0. 
\end{equation}

We shall show that $x^*$ is a fixed point of $T$. For any $n\in \mathbb{N}$, we have

\begin{align*}
	p(x^*,Tx^*) & \leq s [p(x^*,x_{n+1})+p(x_{n+1},Tx^*) ] - p(x_{n+1},x_{n+1}) \\
	& \leq s [p(x^*,x_{n+1})+p(x_{n+1},Tx^*) ]\\
	& \leq s p(x^*,x_{n+1}) + s \lambda p(x^*,x_{n}).
\end{align*}

Using \eqref{eq88} in the above inequality we obtain $p(x^*, T x^*) = 0$, that is, $Tx^*=x^*$.
Thus, $x^*$ is the unique fixed point of $T$. 

\vspace*{0.2cm}

\item[Case 2.] 
If $\max\{ p(x_{n},x_{n-1}), p(x_{n-1},x_{n+1})  \} =p(x_{n+1},x_{n-1}) $, a similar argument as in the Case 1 leads to the existence of a unique fixed point of $T$.

\end{enumerate}

\end{proof}

\begin{problem}
Theorem \ref{main} advocates for the existence of a fixed point for a Chatterjea contraction in a complete partial $b$-metric space for which the constant $s$ is such that $s\geq 2$. An interesting question/problem could be to investigate if Theorem \ref{main} can be formulated for values $1<s<2$ with an appropriate interval for the contraction constant $\lambda.$ Of course Theorem \ref{main} remains true for the sharp inequality $s\geq \sqrt{2}$ but the our question remains since we still have to figure out what happens for $1\leq s < \sqrt{2}.$
\end{problem}

\vspace*{0.5cm}

We conclude this section by presenting a joint Chatterjea-Kannan fixed point leading to the existence of a unique fixed point.
\begin{theorem}\label{main2}
	Let $(X, p)$ be a $0$-complete partial $b$-metric space with coefficient
	$s \geq  1$ and $T:X\to X$ be a self mapping satisfying the following condition:
	\begin{align*}\tag{Ch-Ka}\label{Chat-Kan}
	p(Tx,Ty)&\leq \lambda_1 p(x,y) +\lambda_2 \frac{p(x,Tx)p(y,Ty)}{1+p(x,y)} +\lambda_3 \frac{p(x, T y)p(y, T x)}{1+p(x,y)} \nonumber \\
	&+ \lambda_4 \frac{p(x, T x)p(x, T y)}{1+p(x,y)}+\lambda_5 \frac{p(y, T y)p(y, T x)}{1+p(x,y)}, \nonumber \\
	\end{align*}
	
	for all $x,y \in X$, where $\lambda_1,\lambda_2,\lambda_3,\lambda_4,\lambda_5$ are nonnegative real numbers
	satisfying:
	
	$\lambda_1+\lambda_2+2s\lambda_3+s\lambda_4+s\lambda_5 <1.$ Then $T$ has a unique fixed point $u$ in $X$ and $p(u, u)=0$.
	
%	 $\lambda_1+\lambda_2+\lambda_3+s\lambda_4+s\lambda_5<1.$ 

%Then $T$ and $S$ have a unique common fixed point, i.e. there exists $x^*$ such that $Tx^* = x^* = Sx^*$.
%	
\end{theorem}
In proving this theorem, we shall need the following lemma.

\begin{lemma}\label{lemma2.1}
	Let $(X, p)$ be a partial $b$-metric space with coefficient
	$s \geq  1$ and $T:X\to X$ be a self mapping. Suppose that $(x_n)$ is a sequence in $X$ constructed as $x_{n+1} = Tx_n$ and such that 
	
	\[ p(x_n,x_{n+1}) \leq \lambda p(x_{n-1},x_n),\]
	for all $n\in \mathbb{N},$ where $\lambda\in [0,1)$ is a constant. Then $(x_n)$ is a $0$-Cauchy
	sequence.
\end{lemma}

\begin{proof}
	Let $x_0\in X$ and construct a Picard iterative sequence $(x_n)$ by $x_{n+1} = T x_n, (n \in \mathbb{N})$. We distinguish the following three cases.
	
	\begin{enumerate}
		\item[Case 1.] $\lambda \in [0,\frac{1}{s}) \ (s>1)$. By $p(x_n,x_{n+1}) \leq \lambda p(x_{n-1},x_n)$, we have $p(x_,x_{n+1})\leq \lambda^n p(x_0,x_1).$ Thus, for any $n > m$ and $n, m \in \mathbb{N}$, we have, by following the proof of Theorem \ref{main}
		
		\begin{align*}
			p(x_m,x_n) & \leq s[p(x_m,x_{m+1}) + p(x_{m+1},x_n)]-p(x_{m+1},x_{m+1})  \\
			            & \leq s[p(x_m,x_{m+1}) + p(x_{m+1},x_n)\\
			            & \leq  p(x_m,x_{m+1}) + s^2 p(x_{m+1},x_{m+2}) +s^3 [p(x_{m+2},x_{m+3})+ p(x_{m+3},x_{n})] \\
			            &  \vdots \\
			            & \leq s \lambda^m(1+s\lambda + s^2\lambda^2)+\cdots + s^{n-m-1}\lambda^{n-m-1})p(x_0,x_1)\\
			            & \leq s\lambda^m \left[ \sum_{i=0}^{\infty}(s\lambda)^i \right]p(x_0,x_1)\\
			            & = \frac{s\lambda^m}{1-s\lambda}p(x_0,x_1) \to 0 \ \ (m\to \infty),
		\end{align*}
		
		which implies that $(x_n)$ is a $0$-Cauchy sequence.

		\item[Case 2.] Let $\lambda \in [\frac{1}{s},1) \ (s>1)$. In this case, we have $\lambda^n \to 0$ as $n\to \infty$. So there is $n_o \in \mathbb{N}$ such that $\lambda^{n_o} < \frac{1}{s}$. Thus, by Case 1, we claim that
		
		\[\{ (T^{n_o})x_0\}_{n\geq 1} := \{x_{n_o},x_{n_o+1},\cdots,x_{n_o+n},\cdots \}  \]
		
		is a $0$-Cauchy sequence. Then $(x_n)$ is a $0$-Cauchy sequence.

	\item[Case 3.]	Let $s = 1$. Similar to the process of Case 1, the claim holds.
		
	\end{enumerate}
	
	\end{proof}

Now, we prove the Theorem \ref{main2}.
\begin{proof}
	Choose $x_0\in X$ and construct a Picard iterative sequence $(x_n)$ by $x_{n+1} = T x_n, (n \in \mathbb{N})$. If there exists $n_o \in \mathbb{N}$ such that $x_{n_o} = x_{n_o+1}$ , then $x_{n_o} = x_{n_o+1} = Tx_{n_o}$ i.e. $x_{n_o}$ is a fixed point of $T$. Next, without loss of
	generality, let $x_n \neq x_{n+1}$ for all $n \in \mathbb{N}.$ By \eqref{Chat-Kan}, we have
	
\begin{align*}
p(x_n , x_{n+1} )
&= (T x_{n-1} , T x_n )\\
& \leq \lambda_1 p(x_{n-1},x_{n}) + \lambda_2 \frac{p(x_{n-1},Tx_{n-1})p(x_n,Tx_n)}{1+p(x_{n-1},x_{n})}\\
&+ \lambda_3 \frac{p(x_{n-1},Tx_{n})p(x_n,Tx_{n-1})}{1+p(x_{n-1},x_{n})} + \lambda_4 \frac{p(x_{n-1},Tx_{n-1})p(x_{n-1},Tx_n)}{1+p(x_{n-1},x_{n})}\\
& + \lambda_5\frac{p(x_n,Tx_n)p(x_n,Tx_{n-1})}{1+p(x_{n-1},x_{n})}\\
& = \lambda_1 p(x_{n-1},x_{n}) + \lambda_2 \frac{p(x_{n-1},x_{n})p(x_n,x_{n+1})}{1+p(x_{n-1},x_{n})}
 + \lambda_3 \frac{p(x_{n-1},x_{n+1})p(x_n,x_{n})}{1+p(x_{n-1},x_{n})}\\
 & + \lambda_4 \frac{p(x_{n-1},x_{n})p(x_{n-1},x_{n+1})}{1+p(x_{n-1},x_{n})} + \lambda_5\frac{p(x_n,x_{n+1})p(x_n,x_{n})}{1+p(x_{n-1},x_{n})}.
\end{align*}	

In view of axioms (pm2) and (pm4), we have

\begin{align*}  \lambda_3 \frac{p(x_{n-1},x_{n+1})p(x_n,x_{n})}{1+p(x_{n-1},x_{n})} & \leq \lambda_3\frac{p(x_{n-1},x_{n+1})p(x_n,x_{n})}{p(x_{n-1},x_{n})} \\
&\leq \lambda_3 p(x_{n-1},x_{n+1})\\
&\leq s\lambda_3[p(x_{n-1},x_n)+ p(x_n,x_{n+1})],
\end{align*}
i.e.

\[ \frac{p(x_{n-1},x_{n+1})p(x_n,x_{n})}{1+p(x_{n-1},x_{n})} \leq s\lambda_3[p(x_{n-1},x_n)+ p(x_n,x_{n+1})]. \]

We also have

\begin{align*}
\lambda_4 \frac{p(x_{n-1},x_{n})p(x_{n-1},x_{n+1})}{1+p(x_{n-1},x_{n})} &\leq \lambda_4 \frac{p(x_{n-1},x_{n})p(x_{n-1},x_{n+1})}{p(x_{n-1},x_{n})}\\
& \leq \lambda_4 p(x_{n-1},x_{n+1})\\
&\leq s\lambda_4[p(x_{n-1},x_n)+ p(x_n,x_{n+1})],
\end{align*}

i.e.

\[ \lambda_4 \frac{p(x_{n-1},x_{n})p(x_{n-1},x_{n+1})}{1+p(x_{n-1},x_{n})}  \leq s\lambda_4[p(x_{n-1},x_n)+ p(x_n,x_{n+1})]. \]

Hence 

\begin{align*}
p(x_n , x_{n+1} )
= (T x_{n-1} , T x_n )&\leq \lambda_1p(x_{n-1},x_n)+\lambda_2p(x_n,x_{n+1})+ s\lambda_3[p(x_{n-1},x_n)+ p(x_n,x_{n+1})] \\
 &+ s\lambda_4[p(x_{n-1},x_n)+ p(x_n,x_{n+1})]+\lambda_5 p(x_n,x_{n+1}).
\end{align*}

It follows that 

\begin{equation}\label{eq2.3}
(1-\lambda_2-s\lambda_3-s\lambda_4)p(x_n,x_{n+1}) \leq (\lambda_1+s\lambda_3+s\lambda_4)p(x_{n-1},x_n).
\end{equation}

Again, by \eqref{Chat-Kan}, and exploiting the symmetry of $p$, i.e. $p(x_n,x_{n+1}) = p(Tx_n,Tx_{n-1}) $, we are led to 

\begin{equation}\label{eq2.4}
	(1-\lambda_2-s\lambda_3-s\lambda_5)p(x_n,x_{n+1}) \leq (\lambda_1+s\lambda_3+s\lambda_5) p(x_{n-1},x_n)
\end{equation}

Adding up \eqref{eq2.3} and \eqref{eq2.4} yields

\[ p(x_n,x_{n+1}) \leq \frac{2\lambda_1+2s\lambda_3+s\lambda_4+s\lambda_5}{2-2\lambda_2-2s\lambda_3-s\lambda_4-s\lambda_5} p(x_{n-1},x_n) \]

Put $\lambda =\frac{2\lambda_1+2s\lambda_3+s\lambda_4+s\lambda_5}{2-2\lambda_2-2s\lambda_3-s\lambda_4-s\lambda_5} $. In view of $\lambda_1+\lambda_2+2s\lambda_3+s\lambda_4+s\lambda_5 <1,$ then $0\leq \lambda<1.$

Thus, by Lemma \ref{lemma2.1}, $(x_n)$ is a $0$-Cauchy sequence in $X$. Since
$(X, p)$ is $0$-complete, then there exists some point $x* \in X$ such that:

$$ \lim\limits_{n,m\to \infty} p(x_n , x_m)= \lim\limits_{n\to \infty} p(x_n , x^*)=p(x^*,x^*)=0.$$

 By \eqref{Chat-Kan}, it is easy to see that
 
 \begin{align*}
 	p(x_{n+1},Tx^*) & = p(Tx_n,Tx^*)\\
 	                & \leq \lambda_1p(x_n,x^*) + \lambda_2 \frac{p(x_n,x_{n+1})p(x^*,Tx^*)}{1+p(x_n,x^*)} \\
 	                & + \lambda_3 \frac{p(x_n,Tx^*)p(x^*,x_{n+1})}{1+p(x_n,x^*)} + \lambda_4 \frac{p(x_n,x_{n+1})p(x_n,Tx^*)}{1+p(x_n,x^*)}\\
 	                & + \lambda_5 \frac{p(x^*,Tx^*)p(x^*,x_{n+1})}{1+p(x_n,x^*)}
 \end{align*}

Taking the limit as $n \to \infty$, we get $\lim\limits_{n\to \infty}p(x_{n+1},Tx^*) =0$

On another side, 

\[ p(x^*,Tx^*) \leq s[p(x^*,x_{n+1})+ p(x_{n+1},Tx^*)]-p(x_{n+1},x_{n+1})   \]

Taking the limit on both sides as $n \to \infty$, we get

\[  p(x^*,Tx^*) = 0.\]

It gives that $Tx^* = x^*$ . In other words, $x^*$ is a fixed
point of $T$.

For uniqueness of the fixed point, assume $y^*$ is another fixed point of $T$, then by \eqref{Chat-Kan}, it is easy to check that 

\begin{align*}
p(x^*,y^*) &= p(Tx^*,Ty^*)\\
           & \leq \lambda_1p(x^*,y^*) + \lambda_3 p(x^*,y^*) \\
           & = (\lambda_1+\lambda_3)p(x^*,y^*).
\end{align*}
Because $\lambda_1+\lambda_2+2s\lambda_3+s\lambda_4+s\lambda_5 <1$ implies $\lambda_1+\lambda_3<1,$ we conclude that $x^*=y^*$ since $p(x^*,y^*)=0.$

\end{proof}

\begin{corollary}\label{cor1}
	 Let $(X, p)$ be a complete partial metric space with coefficient
	 $s \geq  1$ and $T:X\to X$ be a self mapping satisfying the following condition:
	 \begin{align*}\tag{Ch-Ka}\label{Chat-Kan}
	 p(Tx,Ty)&\leq \lambda_1 p(x,y) +\lambda_2 \frac{p(x,Tx)p(y,Ty)}{1+p(x,y)} +\lambda_3 \frac{p(x, T y)p(y, T x)}{1+p(x,y)} \nonumber \\
	 &+ \lambda_4 \frac{p(x, T x)p(x, T y)}{1+p(x,y)}+\lambda_5 \frac{p(y, T y)p(y, T x)}{1+p(x,y)}. \nonumber \\
	 \end{align*}
	 
	 for all $x,y \in X$, where $\lambda_1,\lambda_2,\lambda_3,\lambda_4,\lambda_5$ are nonnegative real numbers
	 satisfying:
	 
	 $\lambda_1+\lambda_2+\lambda_3+\lambda_4+\lambda_5 <1.$ Then $T$ has a unique fixed point in $X$.
	 
\end{corollary}

\begin{proof}
	Take $s = 1$ in Theorem \ref{main2}, thus the claim holds.
\end{proof}

\begin{remark}
	Take $\lambda_2 = \lambda_3 = \lambda_4 = \lambda_5=0$ in Theorem \ref{main2} or in Corollary \ref{cor1},
	then Theorem \ref{main2} and Corollary \ref{cor1} are reduced to \cite[Theorem 2.4]{shu} and Banach contraction principle, respectively. From this point of view, our results
	are genuine generalizations of the previous results.
\end{remark}

Recently, Qing and Rhoades \cite{rhodes} established the notion of $T$-stability of
Picard's iteration in metric space. In the following, we modify their definition and introduce the concept
of $T$-stability of Picard's iteration in partial $b$-metric space.

\begin{definition}
	Let $(X, p)$ be a partial $b$-metric space, $x_0 \in X$ and $T : X \to X$ be a
	mapping with $F (T ) \neq \emptyset$, where $F (T )$ denotes the set of all fixed points of $T$. Then Picard's iteration $x_{n+1} = T x_n$ is said to be $T$-stable with respect to $T$ if $x_n \overset{p}{\longrightarrow} q, \ q \in F (T )$ and whenever $(y_n)$ is a sequence in $X$ with $\lim\limits_{n\to \infty} p(y_{n+1},Ty_n)=0$, we have $y_n \overset{p}{\longrightarrow} q$.
\end{definition}

What follows is a useful lemma for the proof of our main result in this section.

\begin{lemma}\cite{liu}\label{useful}
	Let $(a_n)$, $(c_n)$ be nonnnegative sequences satisfying $a_{n+1} \leq ha_n + c_n$ for all $n\in \mathbb{N}, \ 0\leq h<1, \ \lim\limits_{n\to \infty} c_n =0$. Then $\lim\limits_{n\to \infty} a_n =0$. 
\end{lemma}

Now we state our main result on $T$-stability.

\begin{theorem}\label{main3}
	Under the conditions of Theorem \ref{main2}, if $2s\lambda_1 + 2\lambda_3 + (s +
	s^2 )(\lambda_4 + \lambda_5 ) < 2$, then Picard's iteration is $T$-stable.
\end{theorem}

\begin{proof}
	From Theorem \ref{main2}, we know that $T$ has a unique fixed point $x^* \in X$ and $p(x^*,x^*)=0$. Assume that $(y_n)$ is a sequence in $X$ with $\lim\limits_{n\to \infty} p(y_{n+1},Ty_n)=0$. 
	Taking advantage of \eqref{Chat-Kan}, on the one hand, we have
	
	\begin{align*}
		p(Ty_n,x^*) &= p(Ty_n,Tx^*) \\
		            & \leq \lambda_1 p(y_n,x^*) +\lambda_2 \frac{p(y_n,Ty_n)p(x^*,Tx^*)}{1+p(y_n,x^*)}+\lambda_3 \frac{p(y_n,Tx^*)p(x^*,Ty_n)}{1+p(y_n,x^*)}\\
		            & + \lambda_4 \frac{p(y_n,Ty_n)p(y_n,Tx^*)}{1+p(y_n,x^*)}+ \lambda_5 \frac{p(x^*,Tx^*)p(x^*,Ty_n)}{1+p(y_n,x^*)}\\
		            &\leq \lambda_1 p(y_n,x^*) + \lambda_3 p(x^*,Ty_n)+ \lambda_4p(y-n,Ty_n)\\
		            & \leq (\lambda_1+s\lambda_4)p(y_n,x^*) + (\lambda_3 + s\lambda_4)p(x^*,Ty_n),
	\end{align*}
	
	which means 
	
	\begin{equation}\label{eq2.8}
	 (1-\lambda_3 - s\lambda_4) p(Ty_n,x^*) \leq (\lambda_1+s\lambda_4) p(y_n,x^*). 
 \end{equation}

	On the other hand, owing to the symmetry of $p$, we have
	
	$$p(Ty_n,x^*) = p(Tx^*,Tx_n),$$
	
	which yields
	
	\begin{equation}\label{eq2.9}
		1-\lambda_3 - s\lambda_5) p(Ty_n,x^*) \leq (\lambda_1+s\lambda_5) p(y_n,x^*).
			\end{equation}
	
Combining \eqref{eq2.8} and \eqref{eq2.9}, we get 

\[ (2 - 2\lambda_3 - s\lambda_4 - s\lambda_5 )p(x^*,Ty_n) \leq  (2\lambda_1 + s\lambda_4 + s\lambda_5 ) p(x^*,y_n), \]	
	
leading to 

\begin{equation}\label{eq2.10}
p(x^*,Ty_n)\leq \frac{2\lambda_1 + s\lambda_4 + s\lambda_5}{2 - 2\lambda_3 - s\lambda_4 - s\lambda_5}p(x^*,y_n). \end{equation}	
	
	If we set $l = {s(2\lambda_1 + s\lambda_4 + s\lambda_5)}{2 - 2\lambda_3 - s\lambda_4 - s\lambda_5}$, it follows from $2s\lambda_1 + 2\lambda_3 + (s +
	s^2 )(\lambda_4 + \lambda_5 ) < 2$ that $0\leq l<1$.
	
In view of Lemma \ref{useful}, set $a_n = p(y_n,x^*), \ c_n = sp(y_{n+1},Ty_n)$, and owing to \eqref{eq2.10}, we have

\[a_{n+1}= p(y_{n+1},x^*) \leq s[p(y_{n+1},Ty_n) + p(Ty_n,x^*)]\leq ha_n +c_n. \]	

Thus, $\lim\limits_{n\to \infty} p(y_n,x^*)=0=p(x^*,x^*)$, i.e. have $y_n \overset{p}{\longrightarrow} x^*$. As a consequence, Picard's iteration is $T$-stable.
	\end{proof}

\begin{corollary}\label{cor2}
	Under the conditions of Corollary \ref{cor1}, Picard's iteration is $T$-stable.
\end{corollary}	
	
\begin{proof}
	Just notice that Corollary \ref{cor1} is a special case of Theorem \ref{main2} where we take $s=1$. 
	 
\end{proof}	

\begin{corollary}(\cite[Theorem 1.]{shu})\label{cor2}
Let $(X, p)$ be a complete partial $b$-metric space with coefficient
$s \geq  1$ and $T:X\to X$ be a mapping satisfying the following condition:
\begin{equation}\tag{Ch3}\label{Chat3}
p(Tx,Ty)\leq \lambda  p(x,y),
\end{equation}

for all $x,y \in X$, where $\lambda\in \left[0,1 \right)$. The Picard's iteration is $T$-stable.
\end{corollary}

\begin{proof}
	Just notice that Corollary \ref{cor2} is a special case of Corollary \ref{cor1} where we take $\lambda_2=\lambda_3=\lambda_4=\lambda_5=0$. 
	%Then $2s\lambda_1 + 2\lambda_3 + (s +
	%s^2 )(\lambda_4 + \lambda_5 ) < 2$ becomes 
\end{proof}	
\begin{problem}
	The authors plan, in \cite{leko}, to study the $T$-stability of both the Kannan and the Chatterjea contractions for the Picard iteration for a self mapping defined on partial $b$-metric space.
\end{problem}

	The Corollary \eqref{BCPcor} illustrates the idea of the so-called \textit{$P$ property}. If a map $T$ satisfies $F (T ) = F (T^n )$ for each $n\in \mathbb{N}$, then it is said to
	have the $P$ property (see \cite{jeon}). The following results are generalizations of
	the corresponding results in partial $b$-metric spaces.

\begin{theorem}\label{theor2.8}
Let $(X,p)$ be a partial $b$-metric space with coefficient $s\geq 1$. Let $T:X\to X$ be a mapping such that $F(T) \neq \emptyset$ and that 

\begin{equation}\label{eq2.11}
	p(Tx,T^2) \leq \lambda p(x,Tx)
\end{equation}
for all $x\in X$, where $0\leq \lambda<1$ is a constant. Then $T$ has the $P$ property.
\end{theorem}

\begin{proof}
	We always assume that $n > 1$, since the statement for $n = 1$ is trivial. Let $z \in F (T^n )$. It is clear that

	\begin{align*}
	p(z,Tz) \leq p(TT^{n-1}z,T^2T^{n-1}z) &\leq \lambda p(T^{n-1}z,T^nz)= \lambda p(TT^{n-2}z,T^2T^{n-2}z)\\
	& \leq \lambda^2 p(T^{n-2}z,T^{n-1}z)\leq \cdots \leq \lambda^n p(z,Tz) \to 0 \ \ (n\to \infty).
	\end{align*}
Hence, $p(z, T z) = 0$, that is., $T z = z$.	
	
\end{proof}

In concluding this section, we make a conjecture with respect to $P$ property with regards to Theorem \ref{main2} and Corollary \ref{cor1}. They are yet to be proved.

\begin{conjecture}
 	Under the conditions of Theorem \ref{main2}, $T$ has the $P$ property. For the proof, it is enough to check if the mapping $T$ satisfies \eqref{eq2.11}.
\end{conjecture}

Also 

\begin{conjecture}
Under the conditions of Corollary \ref{cor1}, $T$ has the $P$ property.
\end{conjecture}

We conclude this paper by giving examples to illustrate Theorem \ref{main2}.

\begin{example}
Let $X= \{ 1,2,3,4\}$ and $p : X \times X \to \mathbb{R}$ be defined by
\[ 
p(x,y) = 
\begin{cases}
|x-y|^2+\max\{x,y\}, &\text{ if } x\neq y; \\
x, &\text{ if } x=y \neq 1;\\
0,&\text{ if } x=y=1.
\end{cases}
\]

Then $(X, p)$ is a complete partial $b$-metric space with coefficient $s = 4 > 1$.

Now define the self mapping $TX\to X$ by 

\[ T1=1; \ T2=1; \ T3=3; \ T4=2.\]

A simple computation gives:

\[
	\begin{cases}	
		p(T1,T2)= p(1,1) =0 &\leq \frac{3}{4}3 = \frac{3}{4} p(1,2)\\
		p(T1,T3)= p(1,2) =3 &\leq \frac{3}{4} 4 = \frac{3}{4} p(1,3)\\
		p(T1,T4)= p(1,2) =3 &\leq \frac{3}{4}13 = \frac{3}{4} p(1,4)\\

		p(T2,T3)= p(1,2) =3 &\leq \frac{3}{4} 4= \frac{3}{4} p(2,3)\\
		
		p(T2,T4)= p(1,2) =3 &\leq \frac{3}{4}8 = \frac{3}{4} p(2,4)\\
		
		p(T3,T4) = p(2,2) = 2 & \leq \frac{3}{4}5 = \frac{3}{4} p(3,4)\\
		
	\end{cases}
\]
Then, $T$ satisfies all the conditions of Theorem \ref{main2}, with $ \lambda_1 \in \left[\frac{3}{4},1\right), \lambda_2=\lambda_3=\lambda_4=\lambda_5=0 $ and
 obviously $\lambda_1+\lambda_2+2s\lambda_3+s\lambda_4+s\lambda_5 <1.$ Now, by
Theorem \ref{main2}, $T$ has a unique fixed point , which in this case is $1$.
\end{example}

%[\max\{x, y\}]^k
\begin{example}
	Let $X = [0, 1], \ k>1$ and define a mapping $p: X \times X\to \mathbb{R}^+$ by $p(x,y)=   |x - y|^k$ for all $x, y \in X$. Then $(X, p)$ is a complete partial $b$-metric space with coefficient $s = 2^k > 1$. Define a mapping $T:X\to X$ by $Tx = e^{x-\lambda}$, where $\lambda> 1+ \ln 2$ is a constant. Then by mean
	value theorem of differentials, for any $x, y \in X$ and $x \neq y$, there exists some
	real number $\xi$ belonging to between $x$ and $y$ such that $$|e^{x-\lambda}-e^{y-\lambda}|^k = (e^{\xi-\lambda})^k|x-y|^k \leq (e^{1-\lambda})^k|x-y|^k .$$
	
	Hence
	
	\begin{align*}
		p(Tx,Ty) & = |e^{x-\lambda}-e^{y-\lambda}|^k
		\leq (e^{1-\lambda})^k|x-y|^k \leq (e^{1-\lambda})^kp(x,y)  .
	\end{align*}

Then, $T$ satisfies all the conditions of Theorem \ref{main2}, with $ \lambda_1 = (e^{1-\lambda})^k, \lambda_2=\lambda_3=\lambda_4=\lambda_5=0 $ and
obviously $\lambda_1+\lambda_2+2s\lambda_3+s\lambda_4+s\lambda_5 <1.$ Now, by
Theorem \ref{main2}, $T$ has a unique fixed point in $u\in X$.

In view of $\lambda> 1+ \ln 2$, then $\lambda_1 = (e^{1-\lambda})^k < 2^{1-p} = \frac{1}{s}$, so $2s\lambda_1 + 2\lambda_3 + (s +
s^2 )(\lambda_4 + \lambda_5 ) < 2$ and all the conditions of Theorem \ref{main3} are satisfied. So by Theorem \ref{main3}, the Picard's iteration is $T$-stable.

To see exactly what this $T$-stability means, consider the sequence $y_n = \frac{n}{n+1}u \in X$. It follows that

\[p(y_{n+1},Ty_n) = \left| \frac{n+1}{n+2}u - e^{\frac{n}{n+1}u - \lambda} \right| \to |u-e^{u-\lambda}| = 0 \ (n\to \infty). \]
Note that $y_n = \frac{n}{n+1}u\to u \ (n\to \infty).$

\end{example}

\section{Going further}
Recently,  Zheng et al.\cite{zeng} introduced the so-called $\theta$-$\phi$ contraction in complete metric spaces and this technique was successfully applied to Kannan type mapping in partial metric spaces (see \cite{hu}). The results
of the present paper will be applied in future investigations by the authors regarding $\theta$-$\phi$ contraction in complete partial $b$-metric spaces. Hence the continuation of this research is considering $\theta$-$\phi$-Chatterjea type contraction in partial $b$-metric spaces and investigate the existence of fixed points. We have a definition for $\theta$-$\phi$-Chatterjea type contraction
and we must verify that it follows the idea of Chatterjea contractions and generalizes them in a way that keeps their properties
and their relationship with other contractions.
Moreover, a natural question is to check whether this new type of contraction is $T$-stable and has the $P$ property.

\bibliographystyle{amsplain}

\end{document}